\documentclass[preprint,12pt]{elsarticle}

\usepackage{amssymb}
\usepackage{amsthm}

\newcommand{\sysn}{\left\{\begin{array}{rcl}}
\newcommand{\sysk}{\end{array}\right.}

\newtheorem{theorem}{Theorem}[section]

\theoremstyle{example}
\newtheorem{example}[theorem]{Example}

\newtheorem{proposition}[theorem]{Proposition}
\theoremstyle{definition}
\newtheorem{definition}[theorem]{Definition}

\newtheorem{corollary}[theorem]{Corollary}

\newcounter{proposcntr}[section]
\newcounter{theoremcntr}[section]
\newcounter{lemmacntr}[section]

\newcommand{\R}{\mathbb{R}}

\newcommand{\N}{\mathbb{N}}

\begin{document}

\begin{frontmatter}

\title{On homeomorphism groups and the set-open topology}

\author{Alexander V. Osipov}

\ead{oab@list.ru}

\address{Krasovskii Institute of Mathematics and Mechanics, \\ Ural Federal
 University, Ural State University of Economics, Yekaterinburg, Russia}

\begin{abstract}

In this paper we focus on the set-open topologies on the group
${\cal H}(X)$ of all self-homeomorphisms of a topological space
$X$ which yield continuity of both the group operations, product
and inverse function. As a consequence, we make the more general
case of Dijkstra's
 theorem. In this case a homogeneously encircling family $\mathcal{B}$ consists of regular open sets and the closure of every set from
$\mathcal{B}$ is contained in the finite union of connected sets
from $\mathcal{B}$. Also we proved that the zero-cozero topology
of $\mathcal{H}(X)$ is the relativisation to $\mathcal{H}(X)$ of
the compact-open topology of $\mathcal{H}(\beta X)$ for any
Tychonoff space $X$ and every homogeneous zero-dimensional space
$X$ can be represented as the quotient space of a topological
group with respect to a closed subgroup.


\end{abstract}

\end{frontmatter}

\section{Introduction}

Let $X$ be a topological space, ${\cal H}(X)$ the group of all
self-homeomorphisms of $X$ and ${\bf e} : (f,x)\in {\cal
H}(X)\times X \rightarrow f(x)\in X$ the evaluation function.
Since ${\cal H}(X)$ is a group with respect to the usual
composition of functions, we consider those set-open topologies on
${\cal H}(X)$ which yield continuity of both the group operations,
product and inverse function, and, at the same time, yield
continuity of the evaluation function ${\bf e}$. In other words,
we will focus on set-open topologies which make ${\cal H}(X)$ as a
topological group and the evaluation function ${\bf e}$ as a group
action of ${\cal H}(X)$ on $X$. Following \cite{1,21}, a topology
on ${\cal H}(X)$ which makes the evaluation function as a
continuous function is called {\it admissible}. Admissible group
topologies on ${\cal H}(X)$ are those ones which determine a group
action of ${\cal H}(X)$ on $X$.

 In \cite{1} R. Arens proved that if a Hausdorff space $X$ is noncompact, locally compact, and locally connected,
 then $\mathcal{H}(X)$ is a topological group because the compact-open topology coincides with the topology
 that $\mathcal{H}(X)$ inherits from
$\mathcal{H}(\alpha X)$, where $\alpha X$ is a Alexandroff
one-point compactification of $X$.

 In \cite{3} J.J. Dijkstra proved that it is enough
 to required the property: any point of $X$ has a connected
 compact neighbourhood.

\begin{theorem}(Theorem in \cite{3})\label{th11} Let $X$ be a noncompact Hausdorff space. If every
point in $X$ has a neighborhood that is a continuum, then the
compact-open topology on $\mathcal{H}(X)$  coincides with the
group topology that $\mathcal{H}(X)$ inherits from
$\mathcal{H}(\alpha X)$.
\end{theorem}

 A. Di Concilio emphasis that local
 compactness of $X$ is equivalent to the family of compact sets of
 $X$ being a {\it boundedness} of $X$, \cite{hu}, which, jointly any
 EF-proximity of $X$, gives a local proximity space \cite{le}. As a
 consequence, A.Di Concilio make this particular case to fall
 within the more general one in which compact sets are substituted
 with bounded sets in a local proximity space, while the property
 :
 {\it any point of $X$ has a compact connected neighbourhood} is
 replaced by the following one: (*) for each non empty bounded set $B$
 there exists a finite number of connected bounded sets $B_1$,...,
 $B_n$ such that $B<<_{\delta} int(B_1)\cup ... \cup int(B_n)$.

\begin{theorem}\label{th12}(Theorem 5.1 in \cite{adc}) If
$(X,\mathcal{B},\delta)$ is a local proximity space with the
property (*) and any homeomorphism of $X$ preserves both
boundedness and proximity, then the topology of uniform
convergence on bounded sets derived from the unique totally
bounded uniformity associated with $\delta$ is an admissible group
topology on $\mathcal{H}(X)$.

\end{theorem}

 Whenever $(X,\mathcal{B},\delta)$ is a local proximity space,
 then the subcollection of $\mathcal{B}$ of all closed bounded
 subsets of $X$ is a closed, hereditarily closed network of $X$.

We consider the family of $\mathcal{B}$-open topologies which make
${\cal H}(X)$ as a topological group, but the family $\mathcal{B}$
is not necessarily a hereditarily closed network of $X$. As a
 consequence, we make the more general case of Dijkstra's
 Theorem \ref{th11}. In this case a homogeneously encircling family $\mathcal{B}$ consists of regular open sets and the closure of every set from
$\mathcal{B}$ is contained in the finite union of connected sets
from $\mathcal{B}$.

\section{Set-open topology}

The set-open topology is a generalization of the compact-open
topology and the topology of pointwise convergence. This topology
was first introduced by Arens and Dugundji \cite{21}. Now the
set-open topology effective is used in investigation of the
topological properties of function spaces and the groups of all
self-homeomorphisms of topological spaces
\cite{osi4,22,2,25,osi3,osi5,osi6,osi7,osi1,osi2}.

\begin{definition} Let $\mathcal{B}$ be a nonempty family of subsets of a topological space $X$.
 The \textit{$\mathcal{B}$-open topology} on ${\cal
H}(X)$ is a topology generated by the sets of the form $[U,V] =
\{h \in {\cal H}(X) \colon h(U) \subset V\}$ where $U \in
\mathcal{B}$ and $V$ is open. The
\textit{$\overline{\mathcal{B}}$-open topology} is a topology
generated by the sets of the form $[\overline{U},V]$.
\end{definition}

We will denote as ${\cal H}_{\mathcal{B}}(X)$ and ${\cal
H}_{\overline{\mathcal{B}}}(X)$ the set of all homeomorphisms of
$X$ equipped with the $\mathcal{B}$-open topology and
$\overline{\mathcal{B}}$-open topology respectively. In
particular, if $\mathcal{B}$ is a family of all compact subsets of
$X$, then ${\cal H}_c(X)$ is the set of all homeomorphisms of $X$
equipped with the compact-open topology.

Let $\mathcal{B}'$ be a minimal family containing $\mathcal{B}$
and closed under finite unions. Then
$\mathcal{H_{\mathcal{B}'}}(X)$ is homeomorphic to
$\mathcal{H}_{\mathcal{B}}(X)$ and
$\mathcal{H}_{\overline{\mathcal{B}'}}(X)$ is homeomorphic to
$\mathcal{H}_{\overline{\mathcal{B}}}(X)$ so we can always
consider the family $\mathcal{B}$ be closed under finite unions.
It is straightforward that the sets $[U,h(U)]$ with $U\in
\mathcal{B}$ and $h\in \mathcal{H}(X)$ constitute the subbase of
$\mathcal{H_{\mathcal{B}}}(X)$.

 Note that the evaluation function ${\bf e} : (f,x)\in {\cal
H}(X)\times X \rightarrow f(x)\in X$ is continuous in a
$\mathcal{B}$-open topology when $\mathcal{B}$ is a base and it is
continuous in a $\overline{\mathcal{B}}$-open topology when
$\mathcal{B}$ is a base and $X$ is a regular space.

\begin{definition} We will call $\mathcal{B}$ a \textit{homogeneously
encircling family} in $X$ when $\mathcal{B}$ is a family of open
in $X$ sets with following properties.

(a) $h(V) \in \mathcal{B}$ for every $V \in \mathcal{B}$ and every
homeomorphism $h \in \mathcal{H}(X)$;

(b) for every $V, W \in \mathcal{B}$ inclusion $\overline{V}
\subset W$ implies $W\setminus \overline{V} \in \mathcal{B}$;

(c) for every $B\in \mathcal{B}$ there is  $B'\in \mathcal{B}$
such that $\overline{B} \subset B'$.

\end{definition}

 Obviously that if a family $\mathcal{B}$ satisfying condition (a) then
$\mathcal{H_{\mathcal{B}}}(X)$ is a \textit{paratopological
group}, i.e. composition of homeomorphisms is a continuous
operation.

Space $\mathcal{H}_{\overline{B}}(X)$ is a paratopological group
when $\mathcal{B}$ is a homogeneously encircling family satisfying
the next condition:

(d) for every $B,B'\in \mathcal{B}$ such that $\overline{B}\subset
B'$ there is $B''\in \mathcal{B}$ such that $\overline{B}\subset
B''\subset \overline{B''}\subset B'$.

\medskip

\begin{definition} We will call $\mathcal{B}$ \textit{Urysohn
homogeneously encircling family} in $X$ when $\mathcal{B}$ is a
family of open sets satisfying conditions (a), (b), (c) and (d).
\end{definition}

\medskip

It is well known that if $X$ is a compact Hausdorff space, then
$\mathcal{H}(X)$ with the compact-open topology is a topological
group. In general, the compact-open topology does not provide
continuity of the inverse function. For non-compact case, we have
the Theorem  \ref{th11} of Dijkstra for the family $\mathcal{B}$
of all compact subsets of $X$ and the Theorem \ref{th12} of Di
Consilio (a generalization of Dijkstra's theorem) for the family
$\mathcal{B}$ of all closed bounded
 subsets of $X$ where $\mathcal{B}$ is a closed, hereditarily closed network of $X$.

\medskip
Note that there exists a homogeneously encircling family
$\mathcal{B}$ in $X$ such that the $\mathcal{B}$-open topology on
$\mathcal{H}(X)$ is finer than the compact-open topology.

\medskip

\textbf{Example. }Let $X = \{0\} \cup \{\pm 1/n:$ $n \in
\mathbb{N}\}$ with usual topology of real line and let
${\mathcal{B}}$ be the family of all open subsets of $X$. We state
that $\mathcal{H}_{\mathcal{B}}(X) > \mathcal{H}_c(X)$.

Let $V = \{1/n \colon n\in \N\}$ and  ${\bf id}$ is a identical
map. Suppose that ${\bf id} \in \bigcap_{i=1}^p [K_i, U_i] \subset
[V,V]$ for some compact sets $K_i$ and open sets $U_i$. Let number
$j$ be so that $0\in K_i$ for $i<j$ and $0\notin K_i$ for
$i\geqslant j$. There exists number $r\in \N$ such that for every
$n > r$ $\pm 1/n \in U_i$ for $i<j$ and $\pm 1/n \notin K_i$ for
$i\geqslant j$. Let $h$ be a mapping defined as follows: $h(x) =
x$ for any $x \geqslant 1/r$ and $h(x) = -x$ otherwise. It is easy
to see that $h$ is a homeomorphism and $h\in
\bigcap_{i=1}^p[K_i,U_i]$ but $h\notin [V,V]$.

\medskip

A subset $W$ of $X$ is called a regular open set provided
$W=Int_X(cl_X(W))$.

\begin{theorem}\label{TopoGroupCondition} Let homogeneously encircling family $\mathcal{B}$ consists of regular open sets and the closure of every set from
$\mathcal{B}$ is contained in the finite union of connected sets
from $\mathcal{B}$. Then $\mathcal{H}_{\mathcal{B}}(X)$ is a
topological group.
\end{theorem}

\textit{Proof. }It suffices to prove that inversion $h \mapsto
h^{-1}$ is a continuous map. Let $h^{-1} \in [U,h^{-1}(U)]$ where
$U \in \mathcal{B}$. We will find an open neighborhood $V$ of $h$
such that $V^{-1} \subset [U,h^{-1}(U)]$, i.e. for every $g\in V$
and every $x\notin U$ $g(x) \notin h(U)$. First, choose an open
set $C=\bigcup_a C_a \supset \overline{U}$ from family
$\mathcal{B}$ where $\{C_a\}$ is a finite set of disjoint
components. We can find in $\mathcal{B}$ sets $C'$ and $C''$ such
that $\overline{C} \subset C' \subset \overline{C'} \subset C''$.
Let $V = [U,h(U)] \cap [C' \setminus \overline{U}, h(C'\setminus
\overline{U})] \cap \bigcap_a [C_a,h(C_a)] \cap [C''\setminus
\overline{C}, h(C''\setminus \overline{C})]$ and $g\in V$. If
$g(x) \in h(C_a \cap U)$ for some $x\notin C''$ then $g(x)$ is
contained in connected set $h(C_a)$ with $g(C_a)$ because $g\in
V\subset [C_a,h(C_a)]$. Hence $g^{-1}(h(C_a))$ is a connected set
that contains $x$ and $C_a$ which implies that $g^{-1}(h(C_a))
\cap (\overline{C'}\setminus C') \neq \emptyset$. But this fact
means that $g(\overline{C'}\setminus C') \cap h(C_a) \neq
\emptyset$. It is contradiction due to the fact that $g \in V
\subset [C''\setminus \overline{C}, h(C''\setminus \overline{C})]$
and $h(C_a) \subset h(C)$.
 Obviously $g(x) \notin h(U)$ for every $x\in C''\setminus \overline{U}$. We conclude that $g(\overline{U}) = h(\overline{U})$ or
 equally $\overline{g(U)} = \overline{h(U)}$. The latter implies $g(U) = h(U)$ because $U$ is a regular open set.\qed

\begin{theorem}\label{TopoGroupCondition2} Let Urysohn homogeneously encircling
family $\mathcal{B}$ be such that the closure of every set from
$\mathcal{B}$ is contained in the finite union of connected sets
from $\mathcal{B}$. Then $\mathcal{H}_{\overline{\mathcal{B}}}(X)$
is a topological group.
\end{theorem}

\textit{Proof. }The proof is the same as in the
Theorem~\ref{TopoGroupCondition}. \qed

\begin{example} Let $Z= \mathbb{N}\times [0,1]=\bigcup \{\mathbb{I}_n : n\in \mathbb{N}\}$, $X=\alpha Z=Z\cup
\{\alpha\}$ where  $\alpha Z$ is a Alexandroff one-point
compactification of $Z$ and ${\mathcal{B}}$ is the family of all
finite unions of intervals  $\mathbb{I}_n=[0,1]_n$.

Since $h(\alpha)=\alpha$ for any $h\in \mathcal{H}(X)$, then
${\mathcal{B}}$ is a homogeneously encircling family consists of
regular open sets and the closure of every set from $\mathcal{B}$
is contained in the finite union of connected sets from
$\mathcal{B}$. By Theorem \ref{TopoGroupCondition},
$\mathcal{H}_{\mathcal{B}}(X)$ is a topological group. Note that
the point $\alpha$ does not have a neighborhood continuum and
$\mathcal{B}$ is not a hereditarily closed network of $X$.
\end{example}

A topological space $X$ is called semiregular provided that for
each $U\in X$ and $x\in U$ there exists a regular open set $V$ in
$X$ such that $x\in V\subseteq U$.

\begin{proposition}\label{LocCompactSpaceCase} Let $X$ be a locally compact
(semiregular) space and $\mathcal{B}$ consists of all (regular
open) open sets with compact closure. Then
$\mathcal{H}_{\mathcal{B}}(X) \geqslant \mathcal{H}_c(X)$ and
$\mathcal{H}_{\overline{\mathcal{B}}}(X)$ is homeomorphic to
$\mathcal{H}_c(X)$.
\end{proposition}

\begin{proof} Clearly that $\mathcal{B}$ is a Urysohn homogeneously encircling
family. Let $h \in [K, U]$ where $K$ is a compact and $U$ is an
open set. There is a regular open set $V$ such that $K \subset V
\subset \overline{V} \subset h^{-1}(U)$ and $\overline{V}$ is
compact. It immediately follows that $h \in [V,U] \subset [K,U]$
and $h \in [\overline{V},U] \subset [K,U]$.
\end{proof}

\begin{proposition} $\mathcal{H}_{\mathcal{B}}(X) \geqslant
\mathcal{H}_{\overline{\mathcal{B}}}(X)$ for every Urysohn
homogeneously encircling hereditarily open family $\mathcal{B}$.
\end{proposition}

\begin{proof} Let $h \in [\overline{V},U]$ where $V\in
\mathcal{B}$. There exists set $W\in \mathcal{B}$ such that
$W\supset \overline{V}$. Obviously that $h\in [W \cap h^{-1}(U),
U] \subset [\overline{V}, U]$.

\end{proof}

\begin{corollary}(Dijkstra's Theorem
\ref{th11}) Note that if every point from $X$ has a neighborhood
which is continuum. Then $\mathcal{H}_c(X)$ is a topological
group.
\end{corollary}

\begin{proof} Clearly that $X$ is a locally compact space. Let
$\mathcal{B}$ be a family of all open sets with compact closure.
It follows from Proposition~\ref{LocCompactSpaceCase} that
$\mathcal{H}_{\overline{\mathcal{B}}}(X)$ is homeomorphic to
$\mathcal{H}_c(X)$. The fact that every compact set is contained
in a finite union of continua and, by
Theorem~\ref{TopoGroupCondition2}, we have a complete proof of the
corollary.
\end{proof}

\begin{example} Let $X = \R^2$ and $\mathcal{B}$ be a family of all
bounded open sets. For $x\in X$ let the family of the sets
$A_{x,\alpha,\beta} = \{ y \in X \colon \alpha < \arg(y - x) <
\beta \}$ where $\alpha, \beta \in [0, 2\pi]$ be denoted as
$\mathcal{A}_x$. It is straightforward that the family
${\mathcal{B}}_{x_1, \ldots, x_n} = {\mathcal{B}} \cup
\mathcal{A}_{x_1} \cup \ldots \cup \mathcal{A}_{x_n}$ is a Urysohn
homogeneously encircling family and the closure of every set from
${\mathcal{B}}_{x_1, \ldots, x_n}$ is contained in the finite
union of connected sets from ${\mathcal{B}}_{x_1, \ldots, x_n}$.
The latter implies that
$\mathcal{H}_{\overline{{\mathcal{B}}}_{x_1, \ldots, x_n}}(X)$ is
a topological group due to the Theorem~\ref{TopoGroupCondition2}.
It is easy to see that
$\mathcal{H}_{\overline{{\mathcal{B}}}_{x_1, \ldots, x_n}}(X)
\leqslant \mathcal{H}_{\overline{{\mathcal{B}}}_{y_1, \ldots,
y_m}}$ if and only if $\{x_1, \ldots, x_n\} \subset \{y_1, \ldots,
y_m\}$. So we have an infinite lattice of group topologies on
$\mathcal{H}(X)$.
\end{example}

Let $\mathcal{RO}$ be a family of all regular open subsets of $X$.
It is known that $\mathcal{H}_\mathcal{RO}(X)$ is a topological
group (Theorem 4 in \cite{por}). Then for every family
$\mathcal{B}$ satisfying conditions from
Theorem~\ref{TopoGroupCondition}, $\mathcal{H}_{\mathcal{B}}(X)$
is a topological group and $\mathcal{H}_{\mathcal{B}}(X) \leqslant
\mathcal{H}_\mathcal{RO}(X)$.

\section{Metrization condition}

A space $X$ is a $q$-space if for every point $x\in X$ there
exists a sequence $\{U_n : n\in \mathbb{N}\}$ of open
neighbourhoods of $x$ in $X$ such that for every choice $x_n\in
U_n$, the sequence $\{x_n : n\in \mathbb{N}\}$ has a cluster
point.

  D.Gauld and J. van Mill proved that a manifold $M$ is
metrizable if and only if ${\mathcal H}_{c}(M)$ is a $q$-space.

\begin{theorem}(Theorem~4.2 in \cite{4}) For a manifold $M$, the
following statements are equivalent:

\begin{enumerate}

\item $M$ is metrizable;

\item $M$ is separable and metrizable;

\item ${\mathcal H}_{c}(M)$ is first countable;

\item ${\mathcal H}_{c}(M)$ is a $q$-space;

\item ${\mathcal H}_{c}(M)$ is metrizable;

\item ${\mathcal H}_{c}(M)$ is separable and metrizable.

\end{enumerate}

\end{theorem}

We say that $X$ is {\it strongly locally homogeneous} (abbreviated
$\mathrm{SLH}$) if it has a base $\mathcal{B}$ such that for all
$B\in \mathcal{B}$ and $x,y \in B$ there is an element $f\in
\mathcal{H}(X)$ that is supported on $B$ (that is, $f$ is the
identity outside $B$) and moves $x$ to $y$. Clearly every manifold
is $\mathrm{SLH}$.

We are looking for conditions on ${\mathcal
H_{\overline{\mathcal{B}}}}(X)$ and the family $\mathcal{B}$ for a
$\mathrm{SLH}$ space $X$ that ensure that $X$ is metrizable.

Recall that {\it strong development} for $X$ is a sequence of open
covers $\{V_n\}$ such that for every $x\in X$ and every
neighborhood $U$ of $x$ there are open $W$ and number $i$ such
that $x\in W\subset \mathrm{St}(W,V_i) \subset U$.

\begin{theorem}
Let $X$ be $\mathrm{SLH}$ space and let $\mathcal{B}$ be Urysohn
homogeneously encircling base. If ${\mathcal
H_{\overline{\mathcal{B}}}}(X)$ is metrizable topological group
then $X$ is metrizable.
\end{theorem}

\begin{proof} Let $\{U_n\}$ be a sequence of open symmetric neighborhoods of the
identity map $e \in {\cal H_{\overline{\mathcal{B}}}}(X)$ such
that $U_{n+1}^3 \subset U_n$ for every $n$ and $\{U_n\}$ is a
local base at $e$. Let $V_n = \{\mathrm{Int}(\gamma_x(U_n)) \colon
x \in X\}$ where $\gamma_x\colon h \mapsto h(x)$.  We claim that
$\{V_n\}$ is a strong development for $X$ which suffices by the
Moore metrization Theorem.

Claim that $x\in \mathrm{Int}(\gamma_x(U_n))$ for every $x\in X$
and every number $n$.

Let $U_n\supset \bigcap_{i=1}^k [\overline{P_i},Q_i]$ and $x\in
Q_i$ for $i<j$ and $x\notin Q_i$ for $i\geqslant j$. Since $X$ is
$\mathrm{SLH}$, $x$ has an open neighborhood $W$ such that
$W\subset Q_i$ for $i<j$ and $W \cap \overline{P_i} = \emptyset$
for $i\geqslant j$ and for any $y\in W$ there is homeomorphism $h$
supported by $W$ such that $h(x) = y$. So $x\in W\subset
{\mathrm{int}}(\gamma_x(U_n))$.

Let $U = \gamma_x(U_n)$, $W = \gamma_y(U_n)$ and $U\cap W
\neq\emptyset$. For any $w\in W$ there are $z\in U\cap W$ and
$f,g,h \in U_n$ such that $f(x) = z$, $g(y)=z$, $h(y)=w$. The
observation $h \circ g^{-1} \circ f \in U_n^3 \subset U_{n-1}$
implies $\mathrm{ St}(\gamma_x(U_n), V_n) \subset
\gamma_x(U_{n-1})$ i.e. $V_n$ is a star refinement of $V_{n-1}$.
It is suffices to prove that $\bigcup_n V_n$ is a base of $X$. Let
$V$ be an open set and $x\in V$. There is $B \in \mathcal{B}$ such
that $x\in B \subset \overline{B} \subset V$. So we can choose
$U_n\subset [\overline{B},V]$. It is immediate that $x\in
\mathrm{int}(\gamma_x(U_n)) \subset V$.
\end{proof}

\section{Zero-cozero topology}

Let us recall that a~set $F\subseteq X$ is a {\it zero set} if
there exists a continuous function $f\in C(X,\mathbb{R})$ such
that $F=\{x\in X:f(x)=0\}$. A~complement of a~zero set is a~{\it
cozero set}.

 Consider the family $\cal P$ of the sets of the form $[A,U] =
\{h\in \mathcal{H}(X) \colon h(A) \subset U\}$ where $A$ is a
zero-set of $X$ and $U$ is a cozero set of $X$. Let the set
$\mathcal{H}(X)$, equipped with {\it the zero-cozero topology}
generated by $\cal P$, be denoted as $\mathcal{H}_{zc}(X)$. When
$A$ runs over all closed sets in $X$ and $U$ is open in $X$ then
we get {\it the closed-open topology}.

Since any self-homeomorphism of a Tychonoff space $X$ continuously
extends to the Stone- $\check{C}$ech compactification $\beta X$ of
$X$, then $\mathcal{H}(X)$ embeds as a subgroup in
$\mathcal{H}(\beta X)$. Thereby, the relativisation to
$\mathcal{H}(X)$ of the compact-open topology of
$\mathcal{H}(\beta X)$ is an admissible group topology.

\begin{theorem} For a Tychonoff space $X$, $\mathcal{H}_{zc}(X)$ is the relativisation to $\mathcal{H}(X)$ of
the compact-open topology of $\mathcal{H}(\beta X)$.

\end{theorem}

\begin{proof} We can consider the space $\mathcal{H}(X)$ as a subgroup in
$\mathcal{H}(\beta X)$ then we denote the space $\mathcal{H}(X)$
as $\widetilde{\mathcal{H}}(X)$.

 Let $[A,U]$ be a subbase element of $\mathcal{H}_{zc}(X)$ where $A$ is a zero-set of
$X$ and $U$ is a cozero-set in $X$. Let $B:=\overline{A}^{\beta
X}$ and $W=\beta X \setminus \overline{(X\setminus U)}^{\beta X}$.
Clearly that $[B, W]$ is an element of the compact-open topology
of $\mathcal{ H}(\beta X)$. Claim that $[B, W]\cap
\widetilde{\mathcal{H}}(X)$ is homeomorphic to $[A,U]$.

Let $f\in [B, W]\cap \widetilde{\mathcal{H}}(X)$. Then
$f(B)\subseteq W$. Then $h(A)=f(A)\subset U=X\cap W$ where $h\in
\mathcal{H}(X)$ and  $h=f\upharpoonright X$. Hence $h\in [A,U]$.

Let $f\in [A,U]$. Then $f(A)\subseteq U$ and $f(A)$ is a zero-set
of $X$. Since $Z:=X\setminus U$ is a zero-set of $X$,
$\overline{f(A)}^{\beta X}\cap \overline{Z}^{\beta X}=\emptyset$
and $\overline{f(A)}^{\beta X}\subseteq W=\beta X \setminus
\overline{Z}^{\beta X}$. Since any self-homeomorphism of $X$
continuously extends to $\beta X$, there is $\widetilde{f}\in
\widetilde{\mathcal{H}}(X)\subseteq \mathcal{H}(\beta X)$  and
$\widetilde{f}\upharpoonright X=f$. Note that
$\widetilde{f}(A)=f(A)$ and $\widetilde{f}(B)\subseteq
\overline{f(A)}^{\beta X}\subseteq W$. Thus $\widetilde{f}\in
[B,W]\cap \widetilde{\mathcal{H}}(X)$.

\end{proof}

In fact, if $X$ is normal space then the closed-open topology on
${\mathcal{H}}(X)$ coincides with the zero-cozero topology
${\mathcal{H}}(X)$. Then we get the next result of A. Di Concilio.

\begin{corollary}(Theorem 1.8 in \cite{26}) When $X$ is $T_4$, the relativization $\tau_{\beta X}$ to
${\mathcal{H}}(X)$ of the compact-open topology on
${\mathcal{H}}(\beta X)$ is the closed-open topology.
\end{corollary}

\begin{corollary}\label{cor1} If $X$ is a Tychonoff space, then
$\mathcal{H}_{zc}(X)$ with the zero-cozero topology is a
topological group.
\end{corollary}

Important natural examples of homogeneous spaces are provided by
quotients of topological groups with respect to closed subgroups.
Though these quotients need not be topological groups themselves,
they are always homogeneous topological spaces.

A natural question for consideration is the following one
\cite{artk}: which homogeneous spaces can be represented as
quotients of topological group with respect to closed subgroup?  A
partial answer to this question is given below.

\begin{theorem}(N. Bourbaki). Every homogeneous zero-dimensional compact
space $X$ can be represented as the quotient space of a
topological group with respect to a closed subgroup.
\end{theorem}

\begin{proposition}\label{prop1}(L.R. Ford in \cite{for}) If a zero-dimensional $T_1$-space
$X$ is homogeneous, then it is SLH.
\end{proposition}

Suppose that $X$ is a space. A topology $\tau$ on the group
${\mathcal{H}}(X)$ will be called {\it acceptable} if it turns
${\mathcal{H}}(X)$ into a topological group such that the
evaluation function is continuous with respect to the first
variable, and for every open neighbourhood $U$ of the neutral
element $e$ of ${\mathcal{H}}(X)$ and for each $b\in X$, there
exists an open neighbourhood $V$ of $b$ such that $Id_V\subset U$,
where $Id_V:=\{h\in {\mathcal{H}}(X) : h(x)=x$ for each $x\in
X\setminus V \}$.

The next statement is obvious, in view of Corollary \ref{cor1} and
of the definition of acceptable topologies.

\begin{proposition}\label{prop2} For a Tychonoff space $X$, the zero-cozero
topology on the group ${\mathcal{H}}(X)$ is acceptable.
\end{proposition}

\begin{proposition}\label{prop3}(Proposition 3.5.14 in \cite{artk}) Suppose that
$X$ is a homogeneous SLH space, and ${\mathcal{H}}(X)$ is endowed
with an acceptable topology. Then $X$ is canonically homeomorphic
to the quotient space ${\mathcal{H}}(X)/G_a$ where $a$ is a point
of $X$ and $G_a$ is the stabilizer of $a$ in $G$, that is,
$G_a=\{g\in G: g(a)=a\}$.
\end{proposition}

\begin{theorem} Every homogeneous zero-dimensional space $X$
can be represented as the quotient space of a topological group
with respect to a closed subgroup.
\end{theorem}

\begin{proof}
By Proposition \ref{prop1}, the space $X$ is SLH. By Proposition
\ref{prop2}, the zero-cozero topology on the group
${\mathcal{H}}(X)$ is acceptable. Now it follows from Proposition
\ref{prop3} that $X$ is homeomorphic to a quotient space of the
group ${\mathcal{H}_{zc}}(X)$ with respect to a closed subgroup.
\end{proof}

\end{document}